\documentclass[11pt]{amsart}

\usepackage[pdftex]{graphicx}
\usepackage[dvipsnames]{xcolor}

\usepackage[margin=3cm]{geometry}

\usepackage{amsfonts}

\usepackage{amsmath}
\usepackage{amssymb}
\usepackage{amsthm}

\usepackage{paralist}
\usepackage[colorlinks,cite color=blue,pagebackref=true,pdftex]{hyperref}

\usepackage[T1]{fontenc}
\usepackage{parallel}
\usepackage{caption}
\usepackage{subcaption}
\renewcommand\P{P}
\newcommand\dP{\mathbf{P}}
\newcommand\Phat{\widehat{\P}}
\newcommand\Ord[1]{\mathcal{O}(#1)}
\newcommand\TOrd[1]{\Ord{#1}}
\newcommand\FaPa{\mathcal{B}}
\newcommand\Ptop{\widehat{1}}
\newcommand\Pbot{\widehat{0}}
\newcommand\NC{N}
\newcommand\Filter{\mathsf{J}}
\newcommand\Ideal{\mathsf{I}}
\newcommand\1{\mathbf{1}}

\DeclareMathOperator{\sign}{sign}
\DeclareMathOperator{\cone}{cone}
\DeclareMathOperator{\relint}{relint}

\newcommand\R{\mathbb{R}} 
\newcommand\DOrd[1]{\overline{\mathcal{O}}(#1)}
\newtheorem{thm}{Theorem}[section]
\newtheorem{cor}[thm]{Corollary}
\newtheorem{lem}[thm]{Lemma}
\newtheorem{prop}[thm]{Proposition}

\theoremstyle{definition}
\newtheorem{definition}[thm]{Definition}
\newtheorem{example}{Example}

\theoremstyle{remark}
\newtheorem{rem}[thm]{Remark}

\title{Incompatible double posets and double order polytopes}

\author{Aenne Benjes}
\address{Institut f\"ur Mathematik, Goethe-Universit\"at Frankfurt, Germany}
\email{aenne.benjes@stud.uni-frankfurt.de}



\date{\today}

\begin{document}

\maketitle
\begin{abstract}
In $1986$ Stanley associated to a poset the order polytope. The close interplay between its combinatorial and geometric properties makes the order polytope an object of tremendous interest. Double posets were introduced in $2011$ by Malvenuto and Reutenauer as a generalization of Stanleys labelled posets. A double poset is a finite set equipped with two partial orders. To a double poset Chappell, Friedl and Sanyal $(2017)$ associated the double order polytope. They determined the combinatorial structure for the class of compatible double posets. In this paper we generalize their description to all double posets and we classify the $2$-level double order polytopes.
\end{abstract} 
\section{Introduction}
A \textbf{partially ordered set} $(\P,\preceq)$, also called \textbf{poset},
is a finite set $\P$ together with a reflexive, transitive and antisymmetric
relation $\preceq$.  To a poset Stanley ~\cite{twoposet} associates a convex polytope, the \textbf{order polytope}
$\Ord{\P}$, which is the set of all order-preserving
functions from $\P$ into the interval $\left[0,1\right]$:
$$\Ord{\P} = \left\{f:\P \rightarrow \left[0,1\right] \, : \, a \prec b \Rightarrow f(a)\leq f(b)\right\}.$$
Since the order polytope reflects many combinatorial properties of the poset, it is worth to study the geometric properties of $\Ord{\P}$. For more details about convex polytopes we refer to ~\cite{ziegler}.

A \textbf{double poset} $\dP=(\P,\preceq_+,\preceq_-)$, as introduced by
Malvenuto and Reutenauer~\cite{doubleposet}, is a finite set $\P$ together with two
partial order relations $\preceq_+$ and $\preceq_-$. The two underlying posets
are denoted $\P_+ = (\P,\preceq_+)$ and $\P_-=(\P,\preceq_-)$. 
Chappell, Friedl, and Sanyal constructed in ~\cite{twodoubleposet} a polytope
for a double poset $\dP$, the \textbf{double order polytope} given by
$$\TOrd{\dP} \ = \ \TOrd{\P,\preceq_+,\preceq_-} \ := \ \text{conv} \bigl\{ (2\Ord{\P_+} \times \{1\} ) \cup (-2\Ord{\P_-} \times \{-1\} ) \bigr\}\subseteq \R^{\P} \times \R.$$
The interplay of the two partial orders of a double poset is reflected in the geometry of its double order polytope.
The \textbf{reduced double order polytope} is a simpler construction that
captures most properties of $\TOrd{\dP}$, and is defined as 
$$\DOrd{\dP} \ := \ \TOrd{\dP} \cap \{ (f,t) : t = 0 \} \ = \ \Ord{\P_+} - \Ord{\P_-} \subseteq \R^{\P}.$$ 
Note that here and in the following we write $Q - R$ for the Minkowski sum of
polytopes $Q$ and $-R$.
In ~\cite[Thm~2.7]{twodoubleposet} the authors gave a characterization of the
facets of double order polytopes for the class of \emph{compatible} double
posets, that is, the case where $\P_+$ and $\P_-$ have a common linear extension. 
We generalize their description to all
double posets in Theorem~\ref{thm:facets}. 
We use this description to give a complete classification of $2$-level
polytopes among the double order polytopes. 
We finish by determining the vertices of reduced double order polytopes for
general double posets in Corollary~\ref{cor:TO_vertical_edges}.

\textbf{Acknowledgements.} This paper is based on my Bachelor thesis~\cite{benjes}, 
written under the supervision of Prof.~Raman Sanyal. I would like
to thank him for motivation and support while writing the article.

\section{Double posets and double order polytopes}

Let $(\P,\preceq)$ be a poset. By adjoining a new minimum $\Pbot$ and a new
maximum $\Ptop$ to $\P$, we obtain the poset $\Phat$.  The linear form
associated to an order relation $a \prec b$ is the map
$\ell_{a,b}:\R^{\P}\rightarrow \R$ with $$\ell_{a,b}(f) \ := \ f(a) - f(b)$$
for $f \in \R^{\P}$. Moreover, for $a \in \P$ we define $\ell_{a,\Ptop}(f):=
f(a)$ and $\ell_{\Pbot,a}(f) := -f(a)$. With these definitions it follows that
a map $f : \P \to \R$ is contained in $\Ord{\P}$ if and only if
\begin{equation}\label{eqn:ord_poly}
\begin{aligned}
        \ell_{a,b}(f) & \ \le \ 0 \quad \text{ for all } a \prec b,\\
        \ell_{\Pbot,b}(f) & \ \le \ 0 \quad \text{ for all } b  \in \P, \text{
        and}\\
        \ell_{a,\Ptop}(f) & \ \le \ 1 \quad \text{ for all } a  \in \P.
\end{aligned}
\end{equation}

A nonempty \textbf{face} of $\Ord{\P}$ is a subset $F \subseteq \Ord{\P}$ such
that $$ F \ = \ \Ord{\P}^\ell \ := \ \{ f \in \Ord{\P} : \ell(f) \ge \ell(f')
\text{ for all } f' \in \Ord{\P} \}$$ for some linear function $\ell \in
(\R^\P)^*$.  If $F \neq \Ord{\P}$, then $F$ is a proper face.

As mentioned before, the order polytope geometrically describes
combinatorial features of the underlying poset. For example, the vertices of
$\Ord{\P}$ are in bijection to filters of $\P$.  Recall that a \textbf{filter}
of $(\P, \preceq)$ is a subset $\Filter \subseteq \P$ such that $a \in
\Filter$ and $a \prec b$ for $b \in \P$ implies $b \in \Filter$.  Dually, an
\textbf{ideal} is a subset $\Ideal \subseteq \P$  such that $b \in
\Ideal$ and $a \prec b$ for $a \in \P$ implies $a \in \Ideal$. 

For a combinatorial description of faces Stanley ~\cite{twoposet} introduced face partitions.
\begin{definition}\label{def:facepartition}
A (closed) \textbf{face partition} of a face $F \subseteq \Ord{\P}$ is a partition  of $\Phat$ into
nonempty and pairwise disjoint blocks $B_{1},\dots,B_{k}\subseteq \Phat$ such
that 
$$F=\left\{f \in \Ord{\P} : f \, \text{ is constant on }\, B_{i} \, \text{ for } \, i=1,\dots,k\right\}$$ 
and for any $i\neq j$ there is a $f \in F$ such that $f(B_{i})\neq f(B_{j})$. The \textbf{reduced face partition} of $F$ is $\FaPa(F) := \{ B_i : |B_i| > 1 \}$.
\end{definition}

In ~\cite[Prop~2.1]{twodoubleposet} the following description for the normal cone of an nonempty face $F \subseteq \Ord{\P}$ with a reduced face partition $\FaPa(F) = \left\{ B_{1},\dots,B_{k} \right\}$ is given:
\begin{equation}\label{eqn:normalcone}
\NC_\P(F) \ = \ \cone \{\ell_{a,b}: [a,b] \subseteq B_i \text{ for some }i=1,\dots,k \}.
\end{equation}
We will need the following consquences that were noted 
in~\cite{twodoubleposet}.
\begin{cor}\label{cor:maxmin}
    Let $F \subseteq \Ord{\P}$ be a nonempty face with reduced face partition
    $\FaPa = \{ B_1,\dots,B_k \}$. Then for every $\ell \in \relint 
    \NC_\P(F)$ and $p \in \P$ the following hold: 
    \begin{compactenum}[\rm (i)]
    \item if $p \in \min(B_i)$ for some $i$, then $\ell_p > 0$;
    \item if $p \in \max(B_i)$ for some $i$, then $\ell_p < 0$;
    \item if $p \not\in \bigcup_i B_i$, then $\ell_p = 0$.
    \end{compactenum}
\end{cor}

If $P$ is a polytope and $\dim(P)=d$, then we call the $(d-1)$-dimensional
faces \textbf{facets}.  
Maximizing the linear functions $\ell(f,t)=t$ and $\ell(f,t)=-t$ over
$\TOrd{\dP} \subset \R^\P \times \R$ one obtains the facets $2\Ord{\P_+}
\times \{1\}$ and $-2\Ord{\P_-} \times \{-1\}$. We call the remaining facets
\textbf{vertical}. They are  in bijection with the facets of $\DOrd{\dP}$.
A facet of the reduced double order polytope is a face of the form $F = F_+ -
F_-$ such that there is a linear function $\ell \in (\mathbb{R}^{\P})^{*}$,
where $F_+ = \Ord{\P_+}^{\ell}$ and $F_- = \Ord{\P_-}^{-\ell}$. 

\begin{definition}\label{def:rigid}
A linear function $\ell \in (\R^\P)^*$ is called \textbf{rigid} for $\Ord{\P}$ if it satisfies 
\begin{equation}\label{eqn:rigid}
    \relint \NC_{\P_+}(F_+)  \ \cap \
    \relint -\NC_{\P_-}(F_-) \ = \ \R_{>0} \cdot \ell \, 
\end{equation}
for a pair of faces $(F_+,F_-)$. Note that $F = F_+ - F_-$ is necessarily a facet of $\DOrd{\dP}$.
\end{definition}

\begin{definition}\label{def:altchain}
An \textbf{alternating chain} $C$ of a double poset $\dP = (P,\preceq_+, \preceq_-)$ is a
finite sequence of distinct elements
\begin{equation}\label{eqn:alt_chain}
    \Pbot \ = \ 
    p_0 \ \prec_{\sigma} \
    p_1 \ \prec_{-\sigma} \
    p_2 \ \prec_{\sigma} \
    \cdots \ \prec_{\pm \sigma} \
    p_k \ = \ \Ptop,
\end{equation}
where $\sigma \in \{\pm\}$.  If $k$ is odd, then we additionally require that
$p_{k-1} \not\prec_{\sigma} p_1$. For an alternating chain $C$, we define a
linear function $\ell_C$ by 
\[
    \ell_C(f) \ := \ \sigma \, (-f(p_1) + f(p_2) - \dots + (-1)^{k-1} f(p_{k-1})).
\]
If $k=1$, then $\ell_C \equiv 0$. If $k>1$, then $C$ is a \textbf{proper}
alternating chain.
Let $\sign(C)=\tau \in \left\{\pm\right\}$ be the sign of an alternating chain $C$ if $p_{k-1} \prec_\tau p_{k}$ is the last relation in $C$. 
\end{definition}

\begin{definition}\label{def:altcycle}
An \textbf{alternating cycle} $C$ of $\dP$ is a sequence of elements of $\P$ of length
$2k$ of the form
\[
    p_0 \ \prec_{\sigma} \ p_1 \ \prec_{-\sigma} \ p_2 \ \prec_{\sigma} \
    \cdots \ \prec_{-\sigma} \ p_{2k} \
    = \ p_0,
\]
where $\sigma \in \{\pm\}$ and $p_i \neq p_j$ for $0 \le i < j < 2k$. Similarly the linear function associated to $C$ is defined by
\[
    \ell_C(f) \ := \ \sigma (f(p_0) - f(p_1) + f(p_2) - \dots + (-1)^{2k-1} f(p_{2k-1})).
\]
Note that any cyclic shift yields an alternating cycle with the same linear
function $\ell_C$. Hence, we identify an alternating cycle with all its cyclic
shifts.
\end{definition}

\begin{rem}
    Our definition of alternating chains differs slightly from the one given
    in~\cite{twodoubleposet} in that we require $p_{k-1} \not\prec_{\sigma}
    p_1$ for a chain of odd length. Without that condition, alternating cylces
    would yield alternating chains with the same linear function.
\end{rem}

The following technical fact will be of importance later.

\begin{lem}\label{lem:maxvalchain}
If $C$ is a proper alternating chain and $\ell_{C}$ the linear function associated to $C$, then $\max_{f \in \DOrd{\dP}} \ell_{C}(f) = 1$.
More precisely the following hold:
\begin{enumerate}[\rm (i)]
\item if $\sign(C)=+$, then $\max_{f \in \Ord{P_+}}\ell_{C}(f)$=1 and $\min_{f \in \Ord{P_-}} \ell_{C}(f)= 0$;
\item if $\sign(C)=-$, then $\max_{f \in \Ord{P_+}} \ell_{C}(f)=0$ and $\min_{f \in \Ord{P_-}} \ell_{C}(f)= -1$.
\end{enumerate}
\end{lem}
\begin{proof}
    Since the proof works analogously, we only consider the case of an
    alternating chain with $\sign(C)=+$ and odd length:
    $$\Pbot=p_0 \prec_+  p_1 \prec_- p_2 \prec_+ \dots \prec_-  p_{2k} \prec_+
        p_{2k+1} = \Ptop.
    $$ 
    The linear function $\ell_{C}$ associated to $C$ can be written in terms
    of the linear form of the order relation $\preceq_+$:
    $$
        \ell_{C} = \ell_{p_0,p_1} + \ell_{p_2,p_3} + \dots +
        \ell_{p_{2k},p_{2k+1}}.
    $$
    For $f \in \Ord{P_+}$ it follows from \eqref{eqn:ord_poly} that
    $\ell_{p_{2i},p_{2i+1}}(f) \leq 0$ for $0\leq i\leq k-1$ and
    $\ell_{p_{2k},\Ptop}(f) \leq 1$. Hence $\ell_{C}(f) \leq 1$. Let $h$ be
    the smallest even number such that $p_h \prec_+ p_{2k}$ and let $\Filter \subseteq$
    be the principal filter generated by $p_h$. Since $p_{2k} \not
    \prec_+ p_1$ we have $h \geq 2$ and $p_1 \notin \Filter$. Due to the fact that
		$p_{2i} \in \Filter$ implies $p_{2i+1} \in \Filter$ it follows that $\ell_{C}(\1_{\Filter})=1$, and
    hence $\max_{f \in \Ord{P_+}}\ell_{C}(f) = 1$. \newline
We can write $-\ell_{C}(f)$ in terms of the linear form of the order relation $\preceq_-$ as 
$$-\ell_{C} = \ell_{p_1,p_2} + \ell_{p_3,p_4} + \dots + \ell_{p_{2k-1},p_{2k}}.$$
For $f \in \Ord{P_-}$ it holds that $\ell_{p_{2i-1},p_{2i}}(f) \leq 0$ and hence
$\ell_{C}(f) \geq 0$. Because $\ell_{C}(\1_{\emptyset})=0$, $\ell_{C}$ attains
this value.  Hence $\max_{f \in \DOrd{\dP}} \ell_{C}(f) = 1$.
\end{proof}
\begin{lem}\label{lem:maxvalcycle}
Let $C$ be an alternating cycle and $\ell_C$ the linear function associated to $C$. Then $\max_{f \in \DOrd{\dP}}\ell_{C}(f) = 0$.
\end{lem}
\begin{proof}
Let $C$ be the alternating cycle
$$p_0  \prec_{+}  p_1  \prec_{-}  p_2  \prec_{+} \ldots  \prec_{-}  p_{2k} =  p_0.$$
Then we can write the linear function associated to $C$ in terms of the linear form of the order relation $\preceq_+$: 
$$\ell_{C} = \ell_{p_0,p_1} + \ell_{p_2,p_3} + \dots + \ell_{p_{2k-2},p_{2k-1}}.$$
For $f \in \Ord{P_+}$ it follows from \eqref{eqn:ord_poly} that 
$\ell_{p_{2i},p_{2i+1}}(f) \leq 0$ and hence $\ell_{C} \leq 0.$ 
Since $\ell_{C}(\1_{\emptyset}) = 0$ we conclude $\max_{f \in \Ord{P_+}}\ell_{C}(f) = 0$. \newline
Furthermore we can write $-\ell_{C}$ in terms of the linear form of the order relation $\preceq_-$:
$$-\ell_{C} = \ell_{p_1,p_2} + \ell_{p_3,p_4} + \dots + \ell_{p_{2k-1},p_{2k}}.$$
Analogously it follows $\min_{f \in \Ord{P_-}}\ell_{C}(f) = 0$ and thus $\max_{f \in \DOrd{\dP}}\ell_{C}(f) = 0$.
\end{proof}
The following Proposition was stated by Chappell, Friedl and Sanyal in ~\cite{twodoubleposet}.
\begin{prop}
Let $\dP = (\P, \preceq_+, \preceq_- )$ be a double poset. If $\ell$ is a rigid linear function for $\DOrd{\dP}$,
then $\ell = \mu \ell_C$ for some alternating chain or alternating cycle $C$ and $\mu > 0$.
\end{prop}
 
\begin{definition}\label{def:compatible}
A double poset $\dP = (\P,\preceq_+,\preceq_-)$ is called \textbf{compatible} if $\P_+ = (\P,\preceq_+)$ and $\P_- = (\P,\preceq_-)$ have a common linear extension. Otherwise, $\dP$ is \textbf{incompatible}.
\end{definition}
In case $\dP$ is a compatible double poset, it was shown in ~\cite[Thm~2.7]{twodoubleposet} that
the linear functions $\ell_C$ associated to proper alternating chains $C$ are
in bijection to rigid linear functions of $\TOrd{\dP}$. Recall that a linear
extension of $(\P,\preceq)$ is a injective and order-preserving map
$\mathfrak{l} : \P \to [n]$ where $n = |\P|$.

\begin{prop}\label{prop:cycle}
A double poset $\dP = ( \P, \preceq_+, \preceq_-)$ is compatible if and only if it has no alternating cycles.
\end{prop}
\begin{proof}
    If $\dP$ is compatible, then $P_+$ and $P_-$ have a common linear
    extension $\mathfrak{l}: \P \rightarrow \left[n\right]$, where $n = \left|
    \P \right|$.  Suppose there is an alternating cycle 
    $$
        p_0  \prec_{\sigma}  p_1  \prec_{-\sigma}  p_2  \prec_{\sigma}  \dots
        \prec_{-\sigma}  p_{2k} =  p_0.
    $$
    Then $\mathfrak{l}$ has to satisfy $$ \mathfrak{l}(p_0) <
    \mathfrak{l}(p_1) < \mathfrak{l}(p_2) < \dots < \mathfrak{l}(p_{2k-1}) <
    \mathfrak{l}(p_{2k}).$$ Since $p_0 = p_{2k}$ this contradicts
    $\mathfrak{l}(p_0) < \mathfrak{l}(p_{2k})$. 

    Let $\dP$ be a double poset without alternating cycles and
    $\left|\P\right|=n$. Let $M =\max (\P_+)\cap \max (\P_-)$. We claim that
    $M \neq \emptyset$. Otherwise, for every $p \in \max(P_+)$, there is a $q
    \in P \setminus \max(P_+)$ with $p \prec_- q$. And for any such $q$ there
    is a $q' \in \P \setminus \max(P_-)$ with $q \prec_+ q'$. Repeating yields
    an alternating chain or cycle. Since $\left|\P\right|< \infty$ and there
    are no alternating cycles in $\dP$, it has to be a finite sequence, and
    hence there is a $p \in \P$ for which $p \in \max(P_+)$ and $p \in
    \max(P_-)$. 
    We can construct a map $\mathfrak{l} : \P \rightarrow \left\{1, \dots
    ,n\right\}$ that is strictly order preserving for $\prec_+$ and $\prec_-$
    by induction on $n$. For $n=1$, let $\P = \left\{p\right\}$ and
    $\mathfrak{l}(p)=1$. For $n>1$, pick a $p \in M$ and define
    $\mathfrak{l}(p)=n$. By induction, there is a map $\mathfrak{l} : \P
    \setminus \left\{p\right\} \rightarrow \left\{1, \dots ,n-1\right\}$ that
    is strictly order preserving for $\prec_+$ and $\prec_-$. Any map that is
    constructed in this way, gives us a common linear extension for $\P_+$ and
    $\P_-$ and hence $\dP$ is compatible.
\end{proof}

The next example, taken from ~\cite{twodoubleposet}, illustrates that for
incompatible double posets not every alternating chain or cycle corresponds to
a facet of the double order polytope. 

\begin{example}
Let $(\P, \preceq )$ be a poset and $\preceq^{op}$ the opposite order of
$\preceq$. Then $\dP = (\P,\preceq_+,\preceq_-)$ with  $\preceq_+ =
\preceq$ and $\preceq_- = \preceq^{op}$ is an incompatible double poset. Since
$\Ord{\P_+} = \1 -\Ord{\P_-}$, where $\1 : \mathbb{R}^{\P} \rightarrow 
\mathbb{R}$ is the function $\1(p) = 1$ for all $p \in \P$, we conclude,
that the double order polytope is a prism  over $\Ord{\P_+}$. Hence the
vertical facets of $\TOrd{\dP}$ are
prisms over the facets of $\Ord{\P_+}$. Thus the number of facets of
$\DOrd{\dP}$ equals the number of facets of $\Ord{\P_+}$, and these are in
bijection to the minima, maxima, and cover relations of $\P_+$. For any $ p \in
\P$ we have the alternating chains $\Pbot \prec_{+} p \prec_{-} \Ptop$ and
$\Pbot \prec_{-} p \prec_{+} \Ptop$. Furthermore any cover relation $p
\prec_{\sigma} q$ gives rise to the
alternating cycle $p \prec_{\sigma} q \prec_{-\sigma} p$. Hence,  
there are more alternating chains and cycles than facets. 
\end{example}
In the next section, we determine the facets of the reduced double order
polytope for general posets.

\section{Facets and 2-levelness}
Let $\dP = ( \P, \preceq_+, \preceq_-)$ be a double poset.
\begin{definition}\label{def:crossedcycle}
Let $\tau,\sigma \in \left\{\pm\right\}$. An alternating chain or cycle 
$C$. 
is \textbf{crossed} by $a \in \P$ if 
there are $i\neq j$ such that 
$$p_{i} \preceq_\tau a \prec_\tau p_{i+1} \text{ and } p_j \preceq_{\sigma} a \prec_{\sigma} p_{j+1}.$$ 
\end{definition}
The motivation of this definition is the following proposition. 
It was shown in ~\cite[Thm~2.7]{twodoubleposet} that if $\dP$ 
is a compatible double poset, then its alternating chains are 
in bijection to the facets of $\TOrd{\dP}$. To prove it, a 
property of alternating chains of compatible double posets 
is used: \newline
If $p_{i} \prec_\sigma p_{i+1} \prec_{-\sigma} \cdots \prec_{-\tau}
 p_j \prec_\tau p_{j+1}$ is part of an alternating chain $C$ with
 $\sigma,\tau \in \{\pm\}$ and $i < j$, then there is no $a \in \P$
 such that $p_i \prec_\sigma a \prec_\sigma p_{i+1}$ and $p_j 
\prec_\tau a \prec_\tau p_{j+1}$. Uncrossed alternating chains 
and cycles of incompatible double posets fulfil this as well.  
\begin{prop}\label{prop:uncrossed}
If $C$ is an uncrossed alternating chain or cycle, then $\ell_{C}$ is rigid.
\end{prop}
\begin{proof}
We only consider $C$ to be an alternating chain of the form
$$\Pbot=p_0 \prec_+  p_1 \prec_- p_2 \prec_+ \cdots \prec_-  p_{2k} \prec_+  p_{2k+1} = \Ptop,$$
since the proof works analogously for the other forms of alternating chains and cycles. Then the linear function is 
$$\ell_{C}(f) = -f(p_1) + f(p_2) - \dots + f(p_{2k}).$$ Let $F_+ = \Ord{\P_+}^{\ell_C}$ and $F_- = \Ord{\P_-}^{-\ell_C}$ be the corresponding faces. If $\Filter$ is a filter of $\P_+$, then $p_{2i} \in \Filter$ implies $p_{2i+1} \in \Filter$ for $1\leq i \leq k$, since $p_{2i} \prec_+ p_{2i+1}$. It follows from $\sign(C)=+$ with Lemma ~\ref{lem:maxvalchain}(i) that $\max_{\Filter \in P_+} \ell_{C}(\1_\Filter) = 1$.  Thus $\1_\Filter \in F_+$ if and only if $\Filter$ does not separate $p_{2j}$ and $p_{2j+1}$ for $1 \le j \le k$, because otherwise $\ell_C(\1_\Filter) < 1$. From Definition~\ref{def:facepartition} it follows that $p_{2j}$ and $p_{2j+1}$ for $1 \le j \le k$ are not contained in different parts of the face partition $\FaPa_+$.\newline
If $\Filter$ is a filter of $\P_-$, then $p_{2i-1} \in \Filter$ implies $p_{2i} \in \Filter$ for $1\leq i \leq k$, since $p_{2i-1} \prec_- p_{2i}$. It follows again with Lemma ~\ref{lem:maxvalchain}(i) that $\min_{\Filter \in P_-}\ell_C(\1_\Filter) = 0$. Thus a filter $\Filter \subseteq \P_-$ is contained in $F_-$ if and only if $\Filter$ does not separate $p_{2j-1}$ and $p_{2j}$ for $1 \le j \le k$, otherwise $\ell_C(\1_\Filter) > 0$. Again from Definition~\ref{def:facepartition} it follows that $p_{2j-1}$ and $p_{2j}$ for $1 \le j \le k$ are not contained in different parts of the face partition $\FaPa_-$.\newline
Since $C$ is an  uncrossed alternating chain, there is no $a \in \P$ and $i \neq j$ such that $p_{2i} \preceq_+ a \prec_+ p_{2i+1}$ and $p_{2j} \preceq_+ a \prec_+ p_{2j+1}$ and hence there is $f \in F_+$ such that $f(p_{2i})\neq f(p_{2j})$. As well, there is $g \in F_-$ such that $g(p_{2i-1})\neq g(p_{2j-1})$ for any $1 \leq i \le j \leq k$. Thus, the reduced face partitions $\FaPa_\pm$ are \begin{align*}
        \FaPa_+ &\ = \
        \{[p_0,p_1]_{\P_+}, [p_2,p_3]_{\P_+}, \dots, [p_{2k},
        p_{2k+1}]_{\P_+}\}\text{ and}\\
        \FaPa_- &\ = \ \{[p_1,p_2]_{\P_-}, [p_3,p_4]_{\P_-},
        \dots, [p_{2k-1},p_{2k}]_{\P_-} \}.
    \end{align*}
Let $\ell$ be a linear function with $\ell(\phi) = \sum_{p \in \P} \ell_p
\phi(p)$ such that $F_+ = \Ord{\P_+}^{\ell}$ and $F_- = \Ord{\P_-}^{-\ell}$.
Since for $1 \leq i \leq k$ the element $p_{2i}$ is a minimal and $p_{2i-1}$
is a maximal element of $\FaPa_+$, it follows from Corollary~\ref{cor:maxmin}
that $\ell_{p} > 0$ if $p = p_{2i-1}$ and $\ell_{p} < 0$ if $p = p_{2i}$ for
$1 \leq i \leq k$. Since $C$ is an uncrossed alternating chain, it follows
that if $a \in(p_i,p_{i+1})_{\P_+}$ for some $i$, then $a \notin
\left[p_j,p_{j+1}\right]_{\P_-}$ for all $j$ and vice versa. Otherwise there
would be $p_j,p_{j+1}$ such that $p_i \prec_+ a \prec_+ p_{i+1}$ and $p_j
\preceq_- a \preceq_- p_{j+1}$. That is why $a \notin \bigcup_i B_i$ for one
of the face partitions $\FaPa_+$ or $\FaPa_-$ and hence it follows from
Corollary~\ref{cor:maxmin}(iii) that $\ell_{a} = 0$. Since we assumed $F_+ =
\Ord{\P_+}^{\ell}$ and $F_- = \Ord{\P_-}^{-\ell}$, it follows that $\ell
\in \NC_{\P_{+}}(F_{+})$ and $-\ell \in \NC_{\P_{-}}(F_{-})$. As Equation 
~\ref{eqn:normalcone} states we can write
\begin{align*} 
    \NC_{\P_{+}}(F_{+}) & =\cone\left\{ \ell_{p_{0},p_{1}}, \ell_{p_{2},p_{3}},
\dots , \ell_{p_{2k},p_{2k+1}}\right\} \ \text{ and } \\
\NC_{\P_{-}}(F_{-}) & = \cone\left\{ \ell_{p_{1},p_{2}}, \ell_{p_{3},p_{4}},
\dots , \ell_{p_{2k-1},p_{2k}}\right\}.
\end{align*}
So $\ell \in \relint \NC_{\P_+}(F_+) \cap \text{relint} -\NC_{\P_-}(F_-)$ 
satisfies  $\ell_{p_i} + \ell_{p_{i+1}} = 0$ for all $1 \le i \le 2k$ 
and therefore $\ell = \mu \ell_C$ for some $\mu > 0 $.
\end{proof}
\begin{figure} 
  \centering
	\begin{minipage}{.49\textwidth}
	\centering
    \includegraphics[width=0.9\textwidth]{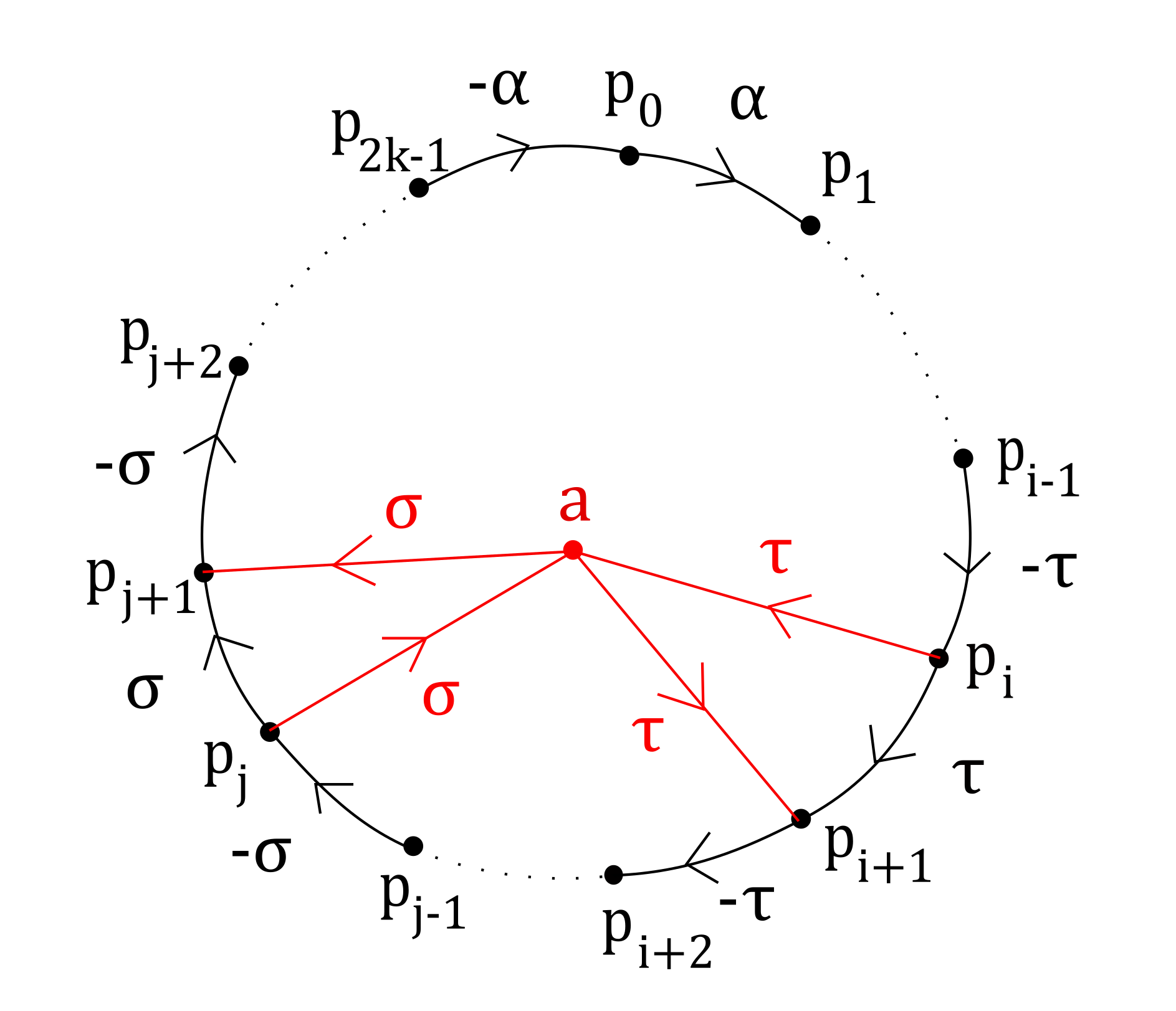}
		\caption{An alternating cycle crossed by a.}
		\label{fig:definitioncycle1}
		\end{minipage}
\end{figure}
The following decomposition of crossed alternating chains and cycles will be important. 
\begin{prop}\label{prop:split}
Let $\dP$ be a double poset. 
\begin{enumerate}[(i)]
\item If $C$ is an alternating cycle crossed by $a$, then there are two alternating cycles $C_1$ and $C_2$ such that $\ell_{C} = \ell_{C_1} + \ell_{C_2}$. 
\item If $C$ is an alternating chain crossed by $a$, then there is a proper alternating chain $C_1$ and an alternating cycle $C_2$ such that $\ell_{C} = \ell_{C_1} + \ell_{C_2}$ and $\sign(C) = \sign(C_1)$. 
\end{enumerate}
\end{prop}
\begin{proof}
(i) Let $C$ be a crossed alternating cycle and $i<j$: 
$$p_{0} \prec_+ \cdots \prec_{-\tau} p_{i} \prec_\tau p_{i+1} \prec_{-\tau} \cdots \prec_{-\sigma} p_j \prec_{\sigma} p_{j+1} \prec_{-\sigma} \cdots \prec_{-} p_{2k}=p_{0}.$$ 
\begin{enumerate}[\rm (1)]
	\item If $\tau = \sigma$, then \newline $p_{0} \prec_+ \cdots \prec_{-\tau} p_{i} \prec_\tau p_{j+1} \prec_{-\tau} \cdots \prec_{-} p_{2k}=p_{0}$ and \newline $p_{i+1} \prec_{-\tau} p_{i+2} \prec_\tau \cdots \prec_{-\tau} p_{j} \prec_\tau p_{i+1}$ are the two alternating cycles $C_1$ and $C_2$.
		\item If $\tau = -\sigma$, then $C_1$ is given by \newline 
		$p_{0} \prec_+ \cdots \prec_{-\tau} p_{i} \prec_\tau a \prec_{-\tau} p_{j+1} \prec_{-\tau} \cdots \prec_{+} p_{2k}=p_{0}$ in case $p_i \neq a$; or \newline
		$p_{0} \prec_+ \cdots \prec_{\tau} p_{i-1}  \prec_{-\tau} p_{j+1} \prec_{\tau} \cdots \prec_{-} p_{2k}=p_{0}$ in case $p_i = a$, and $C_2$ is given by\newline
		$p_i \prec_\tau p_{i+1} \prec_{-\tau} p_{i+2} \prec_\tau \cdots \prec_{\tau} p_{j} \prec_{-\tau} p_i$.
\end{enumerate}
(ii) We only consider the case where $C$ is a crossed alternating chain starting with $+$ and $i<j$: 
$$\hat{0}=p_{0} \prec_+ \cdots \prec_{-\tau} p_{i} \prec_\tau p_{i+1} \prec_{-\tau} \cdots \prec_{-\sigma} p_j \prec_{\sigma} p_{j+1} \prec_{-\sigma} \cdots \prec_{\pm} p_{k}=\hat{1}.$$
\begin{enumerate}[\rm (1)]
\setcounter{enumi}{2}
\item If $\tau = \sigma$, then\newline
			$\hat{0} = p_{0} \prec_+ \cdots \prec_{-\tau} p_i \prec_{\tau} p_{j+1} \prec_{-\tau} \cdots \prec_{\pm} p_k = \hat{1}$ 			is the alternating chain $C_1$ and \newline
			$p_{i+1} \prec_{-\tau} p_{i+2} \prec_{\tau} \cdots \prec_{-\tau} p_{j} \prec_{\tau} p_{i+1}$ 
			is the alternating cycle $C_2$.
\item If $\tau = -\sigma$, then \newline
			$\hat{0} = p_{0} \prec_\alpha \cdots \prec_{-\tau} p_i \prec_{\tau} a \prec_{-\tau} p_{j+1} 
			\prec_{\tau} \cdots \prec_{\pm \alpha} p_k = \hat{1}$ is the alternating chain $C_1$ in case $p_i \neq a$; \newline
			$\hat{0} = p_{0} \prec_\alpha \cdots \prec_{\tau} p_{i-1} \prec_{-\tau} p_{j+1} \prec_{\tau} \cdots \prec_{\pm \alpha} 			p_k = \hat{1}$ is the alternating chain $C_1$ in case $p_i = a$, and \newline 
			$a \prec_\tau p_{i+1} \prec_{-\tau} \cdots \prec_{\tau} p_{j} \prec_{-\tau} a$ 
			is the alternating cycle $C_2$ in both cases.
\end{enumerate} \end{proof}

\begin{figure}
  \centering
  \begin{minipage}{.49\textwidth}
  \centering
    \includegraphics[width=1\textwidth]{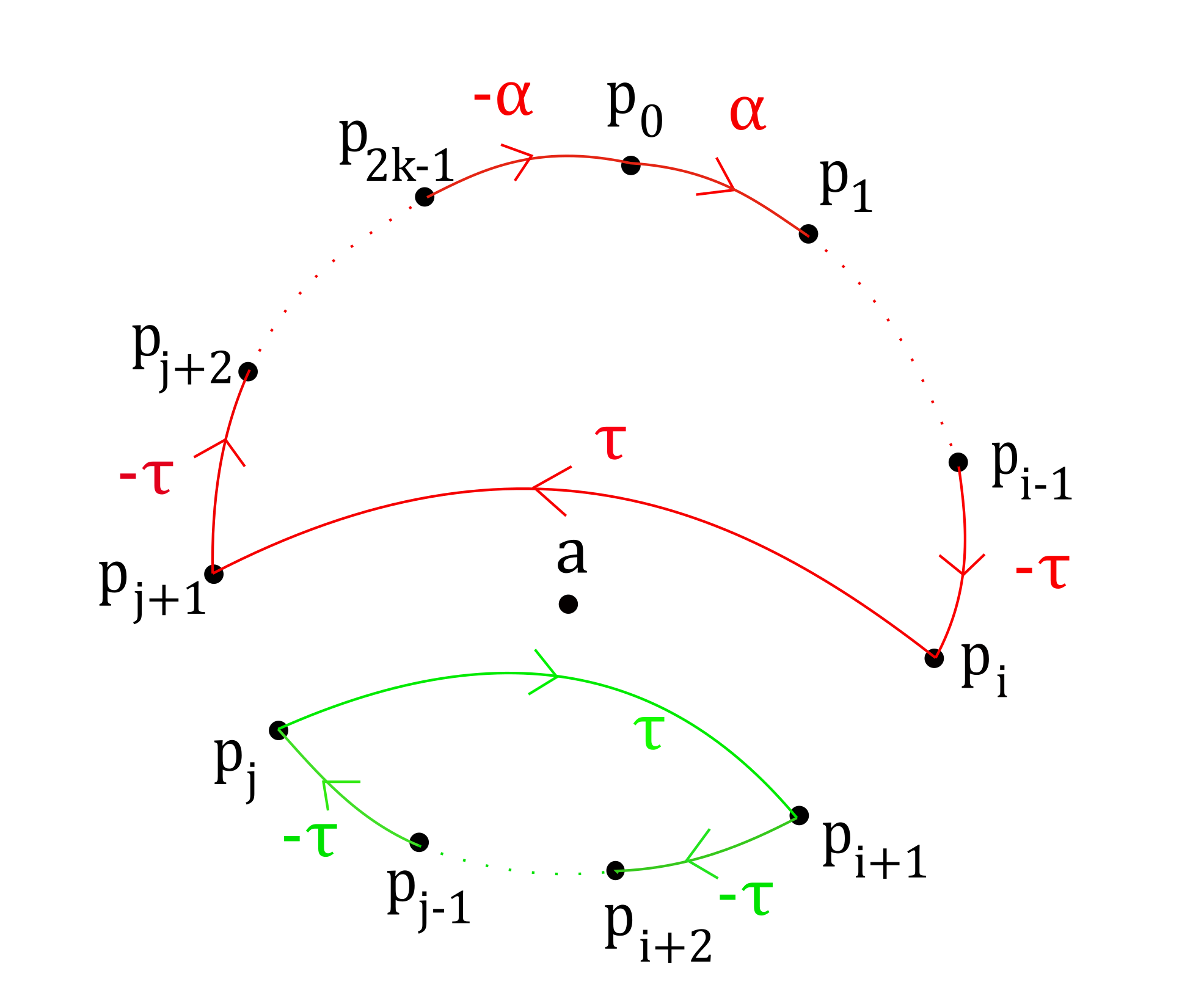}
    \caption{An alternating cycle crossed by a such that $\tau=\sigma$, and two alternating cycles $C_{1}$ (red) and $C_2$ (green). Those satisfy $\ell_{C} = \ell_{C_1} + \ell_{C_2}$. }
		\label{cycle1++}
    \end{minipage}
      \begin{minipage}{.49\textwidth}
  \centering
    \includegraphics[width=1\textwidth]{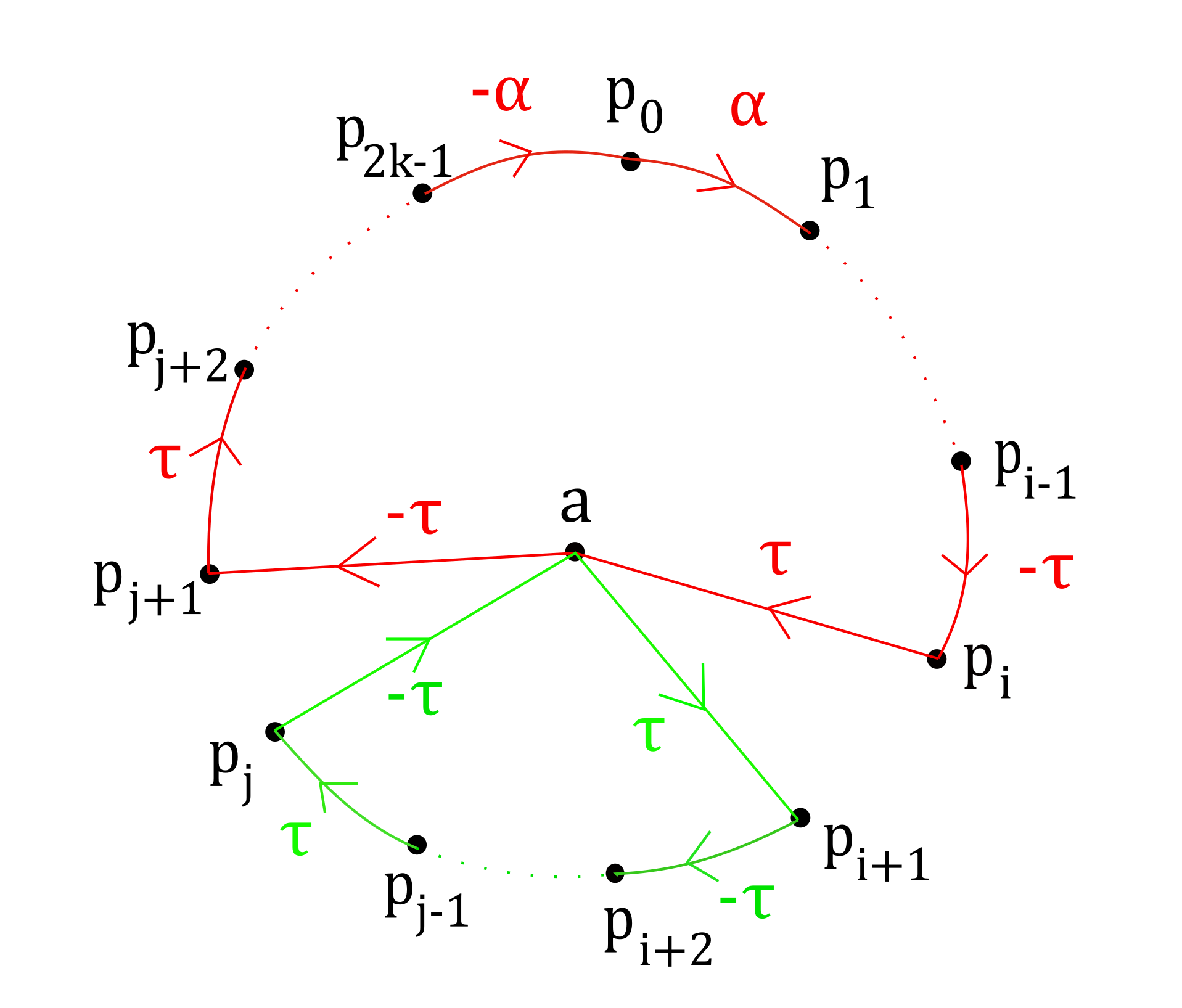}
		\caption{An alternating cycle crossed by a such that $\tau=-\sigma$, and two alternating cycles $C_{1}$ (red) and $C_2$ (green). Those satisfy $\ell_{C} = \ell_{C_1} + \ell_{C_2}$.}
		\label{cycle1+-}
		\end{minipage}
	\end{figure}
	
\begin{cor}\label{cor:crossed}
Let $\dP = (\P,\preceq_+, \preceq_-)$ be a double poset and $C$ an alternating cycle or chain. If there is an $a \in \P$ such that $C$ is crossed by $a$, then $\ell_{C}$ is not rigid.
\end{cor}
\begin{proof}
Assume that $F = \DOrd\dP^{\ell_{C}}$ is a facet. It follows from Proposition~\ref{prop:split}, that there are proper alternating chains or cycles $C_1$ and $C_2$ such that $\ell_{C}= \ell_{C_1} + \ell_{C_2}$ and one of the following holds:
\begin{enumerate}[\rm (i)]
\item $C$, $C_1$ and $C_2$ are alternating cycles;
\item $C$ and $C_1$ are alternating chains that satisfy $\sign(C)=\sign(C_1)$, $C_2$ is an alternating cycle.
\end{enumerate}
Let $G = \DOrd\dP^{\ell_{C_{1}}}$ and $H = \DOrd\dP^{\ell_{C_{2}}}$ be the faces defined by $\ell_{C_1}$ and $\ell_{C_2}$. Let $f \in \text{relint } F$. \newline
In case of (i), since $\ell_{C}(f) = 0$ from Lemma~\ref{lem:maxvalcycle}, this implies $\ell_{C_1}(f) = \ell_{C_2}(f) = 0$.\newline 
In case of (ii), since $\ell_{C}(f) = 1$ from Lemma~\ref{lem:maxvalchain}, this implies $\ell_{C_1}(f) = 1$ and $\ell_{C_2}(f) = 0$. Thus $f \in G \cap H$. Since $f$ was in $\text{relint } F$, it follows that $F \subseteq G \cap H$.
 The alternating chains or cycle $C_1$ and $C_2$ have a length $k>1$ and hence $\ell_{C_{i}}\neq 0$ for $i=1,2$.
Thus $G, \, H$ are proper faces and since we have assumed that $F$ is a facet, it follows that $G$ and $H$ are facets. Since $C_1$ and $C_2$ differ by at least one element it follows, that $\ell_{C_1} \neq \mu \ell_{C_2}$ for all $\mu \in \R_{>0}$ and hence $G \neq H$. Thus $F$ cannot be a facet and hence $\ell_{C}$ is not rigid.
\end{proof}
The following theorem completes the characterization of the facets of double order polytopes and follows from Proposition~\ref{prop:uncrossed} and Corollary~\ref{cor:crossed}.
\begin{thm}\label{thm:facets}
    Let $\dP = ( \P, \preceq_+, \preceq_-)$ be a double poset. A linear function $\ell$ is rigid if
    and only if $\ell \in \R_{>0} \ell_C$ for some uncrossed alternating chain or cycle $C$. In
    particular, the facets of $\TOrd{\dP}$ are in bijection to alternating
    chains and cycles that are not crossed by any $a \in \P$.
\end{thm}

We now turn to the question which incompatible double order polytopes are $2$-level. 
\begin{definition}
A full-dimensional polytope $Q \subseteq\mathbb{R}^{n}$ is \textbf{$2$-level}, if  for every facet-defining hyperplane $H$ there is some $t \in \mathbb{R}^{n}$ such that $H \cup (t + H)$ contains all vertices of $Q$.
\end{definition}
$2$-level polytopes and compressed polytopes~\cite{Sullivant} constitute a very interesting class of polytopes in combinatorics and optimization. In particular Stanleys order polytopes are 2-level and in ~\cite{twodoubleposet}, Chappell, Friedl and Sanyal classified the $2$-level polytopes among compatible double order polytopes. To include the incompatible double order polytopes we need to determine the facet-defining inequalities of
$\TOrd{\dP}$.
\begin{cor}\label{cor:ineq}
Let $\dP = (\P, \preceq_+, \preceq_-)$ be a double poset. Then $\TOrd{\dP}$ is the set of points $(f,t) \in \mathbb{R}^{\P}
\times \mathbb{R}$ such that 
\begin{enumerate}[\rm (i)]
\item $\text{L}_{C} (f,t) := \ell_{C} (f) - \sign(C) \, t \leq 1$ for all uncrossed alternating chains  $C$ of $\P$; 
\item $\text{L}_{C} (f,t) := \ell_{C} (f) \leq  0$ for all uncrossed alternating cycles of $\P$.
\end{enumerate}
\end{cor}
\begin{proof}
Theorem~\ref{thm:facets} says that the facet-defining inequalities of $\TOrd{\dP}$ are in bijection to the uncrossed alternating chains and cycles of $\dP$. If $C$ is an alternating cycle and $\sign{C}=+$, then it follows by Lemma~\ref{lem:maxvalchain} that the maximal value of $\ell_{C}$ over $2 \Ord{\P_+}$ is $2$ and $0$ over $-2 \Ord{\P_-}$. Since the values are exchangend for $\sign{C} = -$, the facet-defining inequalities are of the form $(i)$. If $C$ is an alternating cycle, then it follows by Lemma~\ref{lem:maxvalcycle} that maximal value of $\ell_{C}$ over $2 \Ord{\P_+}$ as well as over $-2 \Ord{\P_-}$ is $0$ and hence the facet-defining inequalities are of the form $(i)$.
\end{proof} 
\begin{prop}\label{prop:level1}
Let $\dP= (\P, \preceq_+, \preceq_-)$ be a double poset and $\sigma \in \left\{\pm\right\}$. If there are $a, b \in \P$ such that $\, \Pbot \prec_{-\sigma} a \prec_{\sigma} b \prec_{-\sigma} \Ptop \, $ is an uncrossed alternating chain $C$ and it does not hold neither $ a \prec_{-\sigma} b$ nor $b \prec_{-\sigma} a$, then $\TOrd{\dP}$ is not $2$-level. 
\end{prop}
\begin{proof}
Since $\TOrd{ \P, \preceq_+, \preceq_-}$ is $2$-level if and only if $\TOrd{\P,\preceq_-,\preceq_+}$ is $2$-level we only consider 
$\sigma = +$. \newline
Due to the fact that $C$ is uncrossed, the linear function $\ell_{C}$ is rigid. 
Then 
\begin{align*}
\text{L}_{C} (f,t)  & = f(a) - f(b) + t 
\end{align*}
is a facet-defining inequality of $\TOrd{\dP}$. Since $b  \not \prec_{-} a$, there is a filter $\Filter_1$ of $\P_-$ such that $b \in \Filter_1$ and $a \notin \Filter_1$. Since $a \not \prec_{-} b$, there is a filter $\Filter_2$ of $\P_-$ such that $a \in \Filter_2$ and $b \notin \Filter_2$. As well, there is a filter $\Filter_3 = \emptyset$ of $\P_-$. 
The vertices corresponding to these three filters let $\text{L}_{C} (f,t)$ take three different values:
\begin{align*}
\text{L}_{C} (-2 \Filter_1,-1) & = 0 - (-2) + (-1) = 1 \\
\text{L}_{C} (-2 \Filter_2,-1) & = -2 - 0 + (-1) = -3 \\
\text{L}_{C} (-2 \Filter_3,-1) & = 0 - 0 + (-1) = -1
\end{align*}
Hence $\TOrd{\dP}$ is not $2$-level. 
For $\sigma=-$, the proof works analogously.
\end{proof}

\begin{thm}\label{thm:2level}
Let $\dP= (\P, \preceq_+, \preceq_-)$ be a double poset and $\sigma \in \left\{\pm\right\}$. Then $\TOrd{\dP}$ is $2$-level if and only if for all $a, b \in \P$ such that $a \prec_{\sigma} b$ is part of an uncrossed alternating chain or cycle it holds that $a \prec_{-\sigma} b$ or $b \prec_{-\sigma} a$.
\end{thm}
\begin{proof}
Again, we consider only $\sigma = +$. For $\sigma = -$, the proof works analogously.\newline
If $b \prec_- a$, then $a \prec_+ b$ can only be part of the alternating cycle 
$$C = a \prec_+ b \prec_- a.$$ All other alternating chains or cycles would be crossed by $a$. The corresponding linear function of the double order polytope 
\begin{align*}
\text{L}_{C}(f,t) & = f(a) - f(b)
\end{align*}
defines a facet of $\TOrd{\dP}$. If $\Filter_+$ is a filter of $\P_+$, then $a \in \Filter_+$ implies $b \in \Filter_+$ 
and that is why $\text{L}_{C}(2 \1_{\Filter_+},1)=0$ or $\text{L}_{C}(2 \1_{\Filter_+},1)=-2$. If $\Filter_-$ is a filter of $\P_-$, then $b \in \Filter_-$ implies $a \in \Filter_-$ and that is why $\text{L}_{C}(-2 \1_{\Filter_+},-1)=0$ 
or $\text{L}_{C}(-2 \1_{\Filter_+},-1)=-2$.\newline
If $a \prec_- b$, then $a \prec_+ b$ can be part of an alternating chain or cycle $C'$ such that $C' \neq C$. In this case all other $c \prec_\tau d$ in $C'$ have to satisfy $c \prec_{-\tau} d$, where $\tau \in \left\{\pm\right\}$. Otherwise, if $d \prec_{-\tau} c$, then $C'$ would be crossed by $c$. Hence $C'$ is an alternating chain. Let $C'$ be the alternating chain 
$$\Pbot = p_0 \prec_{\tau} p_1 \prec_{-\tau} \dots \prec_{\pm} p_k = \Ptop.$$ If $\Filter$ is a filter of $P_{+}$ or $P_{-}$, then it follows from $p_i \in \Filter$ that $p_{i+1} \in \Filter$, since $p_i \prec_+ p_{i+1}$ and $p_{i} \prec_- p_{i+1}$ for $i = 0, \dots , k-1$. Let $\sign(C')=+$. 

If $\Filter_+ \subseteq P_+$, then $\ell_{C'}(2 \1_{\Filter_+})$ can only take the values 2 or 0 and if $\Filter_- \subseteq P_-$, then $\ell_{C'}(-2 \1_{\Filter_-})$ takes the values 0 and -2. The values are exchanged for $\sign(C')=-$. Hence 
\begin{align*}
\text{L}_{C'}(f, t ) & = \ell_{C'}(f) - \sign(C') t
\end{align*}
where $(f,t)$ is a vertex of $\TOrd{\dP}$ attains only the values -1 and 1. 
Thus $\TOrd{\dP}$ is $2$-level.
 \newline
\newline
Assume that $\TOrd{\dP}$ is $2$-level. If there are $a,b \in \P$ such that $a \prec_{\sigma} b$ is part of an uncrossed alternating chain or cycle and neither $a \prec_{-\sigma} b$ nor $b \prec_{-\sigma} a$, then it follows 
by Proposition~\ref{prop:level1} that $\TOrd{\dP}$ is not $2$-level.
\end{proof}

\section{Edges of general double order polytopes}
In this last section we determine the vertical edges of double order polytopes. The edges of an order polytope $\Ord{\dP}$ were determined by Stanley ~\cite{twoposet}: Edges correspond to pairs of filters $\Filter \subset \Filter'$ such that 
$\Filter' \setminus \Filter$ is a connected poset.
The vertical edges of $\TOrd{\dP}$ are in bijection to the vertices of $\DOrd{\dP}$ and the following theorem shows 
that they also correspond to certain pairs of filters $(\Filter_+, \Filter_-)$ where $\Filter_+ \subseteq \P_+$ and 
$\Filter_- \subseteq \P_-$.

\begin{thm}\label{cor:TO_vertical_edges}
    Let $\dP = (P,\preceq_+,\preceq_-)$ be a double poset and let $\Filter_+
    \subseteq \P_+$ and $\Filter_- \subseteq \P_-$ be filters.  Let $\Ideal_+
    := \P_+ \setminus \Filter_+$  and $\Ideal_- := \P_- \setminus \Filter_-$
    be the corresponding ideals.
    Then $(2\1_{\Filter_+},1)$ and $(-2\1_{\Filter_-},-1)$ are the endpoints
    of a vertical edge of $\TOrd{\dP}$ if and only if $\1_{\Filter_+} -
    \1_{\Filter_-}$ is a vertex of $\DOrd{\dP}$ if and only if the following
    hold:
    \begin{enumerate}[\rm (i)]
        \item for all $a \in \Filter_+ \cap \Filter_-$ there is an alternating
            chain 
            $$
               \Pbot \ \prec_{-\sigma} \ a_1 \ \prec_\sigma \ a_2 \
               \prec_{-\sigma} \ \cdots \ \prec_\pm \ a_k \ = \ a \ \prec_\mp \
               \Ptop,
            $$
            where $a_1 \in 
           \Filter_\sigma \setminus \Filter_{-\sigma}$ and 
            $a_2, \dots, a_{k} \in \Filter_+ \cap \Filter_-$.
        \item for all $b \in \Ideal_+ \cap \Ideal_-$ there is an alternating
            chain 
            $$
               \Pbot \ \prec_\pm \ b \ = \ b_1 \ \prec_\mp \ b_2
            \ \prec_{\pm} \ \cdots \ \prec_\sigma \ b_k \ \prec_{-\sigma} \Ptop,
            $$
            where $b_1, b_2, \dots, b_{k-1} \in \Ideal_+ \cap \Ideal_-$  and
            $b_k \in 
           \Ideal_\sigma \setminus \Ideal_{-\sigma}$.
    \end{enumerate}
		This generalizes the result of Chappell, Friedl and Sanyal in Corollary $2.17$~\cite{
		twodoubleposet}, since (i) implies that $\min{\Filter_+} \cap \min{\Filter_-} = \emptyset$
		and (ii) implies that $\max{ \P_+ \setminus \Filter_+ } \cap \max{ \P_- \setminus \Filter_-}
		= \emptyset$.
\end{thm}
\begin{proof}
		
    \newcommand\CC{\mathcal{C}}%
    From the definition of the reduced double order polytope 
    \[
        \TOrd{\dP} \cap \{(\phi,t) : t = 0\} \ = \ 
        (\Ord{\P_+}-\Ord{\P_-})  \times \{0\}
    \]  
    and the fact that $\1_{\Filter_+} - \1_{\Filter_-}$ is the midpoint
    between $(2\1_{\Filter_+},1)$ and $(-2\1_{\Filter_-},-1)$ the first
    equivalence follows.

    To show necessity, assume that (i) is violated for some element $a \in
    \Filter_+ \cap \Filter_-$. Let $\CC$ be the union of all alternating
    chains
    \begin{equation}\label{eqn:XX}
        \Pbot \ \prec_{-\sigma} \ a_1 \ \prec_\sigma \ a_2 \ \prec_{-\sigma} \
        \cdots \ \prec_\pm \ a_k \ = \ a \ \prec_\mp \ \Ptop,
    \end{equation}
    such that $a_1,\dots,a_k \in \Filter_+ \cap \Filter_-$. 
    
    We claim that $\Filter_+ \setminus \CC$ is a filter in $P_+$. 
    Otherwise there is an element $a_0 \in \Filter_+ \setminus \CC$ and an
    element $a_1 \in \CC$ such that $a_0 \prec_+ a_1$. Since $a_1 \in \CC$,
    there is an alternating chain of the form~\eqref{eqn:XX}. We can assume
    that $\sigma=-$. Otherwise, $a_0 \prec_+ a_2$ and we simply delete $a_1$
    from the  alternating chain. By construction $a_0 \in \Filter_+ \setminus \Filter_-$
    and the alternating chain $a_0 \prec_+ a_1 \prec_- \cdots \prec_\pm a_k = a$
    contradicts our assumption. 
    
    \newcommand\op{\mathrm{op}}%
    The same argument yields that $\Filter_- \setminus \CC$ is a filter in
    $P_-$.  Thus $\1_{\Filter_+} - \1_{\Filter_-} = \1_{\Filter_+ \setminus
    \CC} - \1_{\Filter_- \setminus \CC}$ and therefore $\1_{\Filter_+} -
    \1_{\Filter_-}$ cannot be a vertex of $\DOrd{\dP}$. The same argument
    shows necessity of (ii). Indeed, let us write $\P^\op$ for the poset $\P$
    with the opposite order relation. Filters of $\P^\op$ are ideals in $\P$
    and conversely and $\Ord{\P^\op} = \1 - \Ord{\P}$. In
    particular $\Ord{\P_+^\op} - \Ord{\P_-^\op} = \Ord{\P_-} - \Ord{\P_+} =
    -\DOrd{\dP}$. Since $\1_{\Filter_+} - \1_{\Filter_-}$ is vertex of
    $\DOrd{\dP}$ if and only if $\1_{\Filter_-} - \1_{\Filter_+} =
    \1_{\Ideal_+} - \1_{\Ideal_-}$ is a vertex of $-\DOrd{\dP}$. Condition
    (ii) is identical to condition (i) for the opposites of $P_+$ and $P_-$.

    For sufficiency, let $a \in \min{\Filter_+}$. If $a \in \Filter_+
    \setminus \Filter_-$, then set $\ell_{+a}(f) := f(a)$.
    If $a \in \Filter_+ \cap \Filter_-$, then let 
    \begin{equation}\label{eqn:YY}
        \Pbot \ \prec_{-\sigma} \ a_1 \ \prec_\sigma \ a_2 \ \prec_{-\sigma} \
        \cdots \ \prec_- \ a_k \ = \ a \ \prec_+ \ \Ptop,
    \end{equation}
    be a chain $C$ as in (i). Note $\sign(C) = +$ since $a \in \min
    \Filter_+$. Lemma~\ref{lem:maxvalchain}(i) yields that
    $\ell_{+a}(\1_{\Filter'_+}) \le 1 = \ell_{+a}(\1_{\Filter_+})$ for every
    filter $\Filter'_+ \subseteq \P_+$. Moreover, if
    $\ell_{+a}(\1_{\Filter'_+}) = 1$, then $a \in \Filter'_+$. Again by
    Lemma~\ref{lem:maxvalchain}(i), we have $\ell_{+a}(-\1_{\Filter'_-}) \le 0
    =\ell_{+a}(-\1_{\Filter_-})$ for all filter $\Filter'_- \subseteq \P_-$.
    Analogously, we use (ii) and define $\ell_{+b}$ for all $b \in \max
    P_+\setminus\Filter_+$. We set
    \[
        \ell_+(f) \ := \ 
        \sum_{a \in \min \Filter_+} \ell_{+a}(f)  + 
        \sum_{b \in \max \P_+ \setminus \Filter_+} \ell_{+b}(f) 
        \, .
    \]
    Then $\ell_+$ is maximized over $\DOrd{\dP}$ at points $\1_{\Filter_+} -
    \1_{\Filter'_-}$ for some $\Filter'_- \subseteq \P_-$. Importantly,
    $\1_{\Filter_+} - \1_{\Filter_-}$ is one of the maximizers.

    The same construction applied to $\Filter_-$ yields a function $\ell_-(f)$
    which is maximized over $\DOrd{\dP}$ at points $\1_{\Filter'_+} -
    \1_{\Filter_-}$ for some $\Filter'_+ \subseteq \P_+$.  Again,
    $\1_{\Filter_+} - \1_{\Filter_-}$ is one of the maximizers. It follows
    that the linear function $\ell_+ + \ell_-$ is uniquely maximized
    $\1_{\Filter_+} - \1_{\Filter_-}$ over $\DOrd{\dP}$.
\end{proof}

\bibliography{article}
\bibliographystyle{plain}

\end{document}